\documentclass[12pt]{article}

\usepackage{amssymb}
\usepackage{amsfonts}
\usepackage{amsmath}
\usepackage{amsthm}
\usepackage{graphicx}
\usepackage[boxruled]{algorithm2e}

\def\Rs{{\mathbb{R}^3}}

\def\bs{{{\bigskip}}}
\def\o{{\omega}} 
 
\def\a{{\alpha}}
\def\b{{\beta}}
\def\g{{\gamma}}

\def\d{{\delta}}
\def\D{{\Delta}}
 
\def\q{{\theta}}

\def\t{{\tau}}

\newtheorem{thm}{Theorem} 
\newtheorem{lem}{Lemma} 
\newtheorem{prop}[thm]{Proposition}

\usepackage{xcolor}

\usepackage{enumitem}
\usepackage{url}
\newcommand{\figlab}[1]{\label{fig:#1}}
\newcommand{\seclab}[1]{\label{sec:#1}}

\newcommand{\lemlab}[1]{\label{lem:#1}}
\newcommand{\thmlab}[1]{\label{thm:#1}}
\newcommand{\figref}[1]{\ref{fig:#1}}
\newcommand{\secref}[1]{\ref{sec:#1}}

\newcommand{\lemref}[1]{\ref{lem:#1}}
\newcommand{\thmref}[1]{\ref{thm:#1}}
\newcommand{\hide}[1]{}
\title{Simple Closed Quasigeodesics\\ on Tetrahedra}

\author{
Joseph O'Rourke%
    \thanks{Department of Computer Science, Smith College, Northampton, MA
      01063, USA.
      \protect\url{jorourke@smith.edu}.}
\and
Costin V\^{i}lcu%
    \thanks{``Simion Stoilow'' Institute of Mathematics  
      of the Romanian Academy,
      P.O. Box 1-764,
      RO-014700 Bucharest, Romania.
    \protect\url{Costin.Vilcu@imar.ro}.
    }
}

\date{\today}

\begin{document}
\maketitle

\begin{abstract}
Pogorelov proved in $1949$ that every every convex polyhedron has at least three simple closed quasigeodesics.  
Whereas a geodesic has exactly $\pi$ surface angle to either side
at each point, a \emph{quasigeodesic} 
has at most $\pi$ surface angle to either side at each point.
Pogorelov's existence proof did not suggest a way to identify the three 
quasigeodesics, and it is only recently that a finite algorithm has been proposed.

Here we identify three simple closed quasigeodesics on any tetrahedron:
at least one through 
$1$ vertex, at least one through $2$ vertices, and at least
one through $3$ vertices.
The only exception is that isosceles tetrahedra 
have simple closed geodesics 
but do not have a $1$-vertex quasigeodesic.

We also identify an infinite class of tetrahedra 
that each have at least $34$ simple closed quasigeodesics.
\end{abstract}


\section{Introduction}
\seclab{Introduction}
It is well-known that every convex polyhedron has at least three simple closed quasigeodesics~\cite{p-qglcs-49},
a counterpart to the Lusternik-Schnirelmann theorem that every smooth closed convex
surface has at least three simple closed geodesics.
Whereas a geodesic on a convex polyhedron has exactly $\pi$ surface angle to either side
at each point, a
\emph{quasigeodesic} 
has at most $\pi$ surface angle to either side of any point.
Unlike geodesics, quasigeodesics can pass through vertices.

As Pogorelov's result does not lead directly to an algorithm,
it was posed as an open problem to find a polynomial-time algorithm to construct
at least one simple closed quasigeodesic~\cite[Open Prob.~24.2]{do-gfalop-07}.
Even a finite algorithm was not known.
Recently there has been progress on this question~\cite{demaine2020finding},
and an exponential-time algorithm
has been developed~\cite{ChartierArnaud}.

In this paper we describe the three quasigeodesics guaranteed by Pogorelov, in the particular case of tetrahedra.

In~\cite{Reshaping} we conjectured that every convex polyhedron has
either a simple closed geodesic, or a simple closed quasigeodesic through
exactly one vertex. 
We proved this conjecture for doubly-covered convex polygons~\cite[Ch.~17]{Reshaping}.
Here we prove it for all tetrahedra.

\begin{thm}
\thmlab{Q123}
Every tetrahedron has 
a $2$-vertex quasigeodesic, a $3$-vertex quasigeodesic, and
a simple closed geodesic or a $1$-vertex simple closed quasigeodesic.
\end{thm}

Our result complements in some sense the work of Protasov~\cite{protasov2007closed}, which determines closed geodesics on simplices.

Ellipsoids are well-known examples of smooth surfaces which admit only three simple closed geodesics.
Our second result establishes that many tetrahedra have an unexpected wealth of simple closed quasigeodesics.

\begin{thm}
\thmlab{34}
There exists an open set ${\cal O}$ in the space of all tetrahedra, each element of which has at least $34$ simple closed quasigeodesics.
\end{thm}

All our proofs are constructive and lead to algorithms, constant-time in an appropriate model of computation. See~\cite{demaine2020finding}
for a discussion of models of computation for quasigeodesics.

After presenting our proofs, we conclude the paper with a short section of
remarks and open questions.


\subsection{Notation}
Here we list basic notation that we use throughout.
More specialized notation and preliminaries will be introduced where needed.

\begin{figure}[htbp]
\centering
\includegraphics[width=1.0\linewidth]{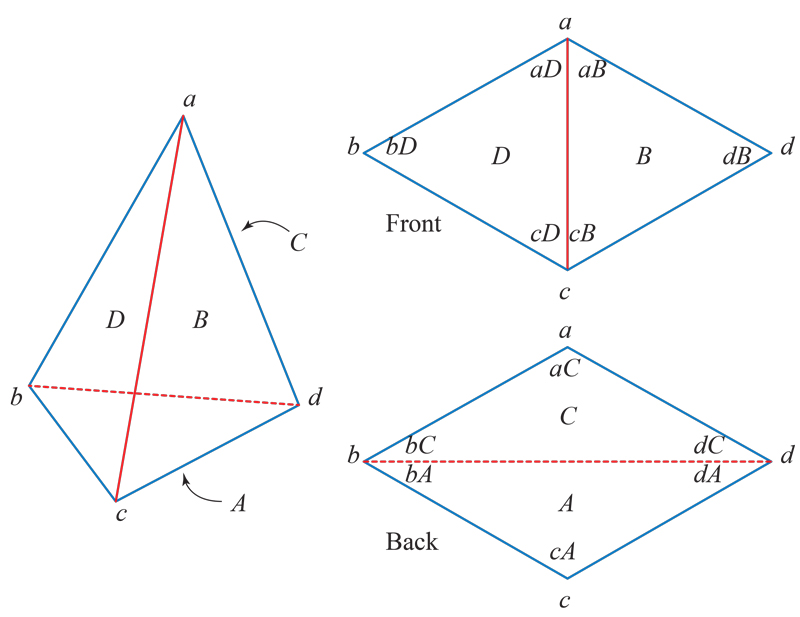}
\caption{$A=bdc$, $B=cda$, $C=adb$, $D=abc$.}
\figlab{NotationAngles}
\end{figure}
\begin{itemize}
\item Vertices of tetrahedron $T$: $a,b,c,d$.
\item Face $A$ is opposite $a$. 
So: $A=bdc$, $B=cda$, $C=adb$, $D=abc$. 
\item Face angles are specified by vertex and face.
So the three face angles incident to vertex $a$ are: $aB, aC, aD$; etc.
See Fig.~\figref{NotationAngles}.
\item Complete angle at $a$: $\q_a=aB +  aC + aD$.
\item Vertex curvature at $a$: $\omega_a = 2 \pi - (aB +  aC + aD)$.
\end{itemize}
For succinctness, we will often use the symbol $Q_k$ 
as shorthand for a ``$k$-vertex simple closed quasigeodesic."


\section{$Q_0$: Simple Closed Geodesics}
\seclab{Q0}
We use the Gauss-Bonnet theorem in 
two forms:
\begin{enumerate}
\item The total curvature at the four vertices sums to $4\pi$.
\item The turn $\t$ of a closed curve plus the curvature enclosed equals $2\pi$:
$\t+\o=2\pi$.
\end{enumerate}

By the first form of Gauss-Bonnet,
a simple closed geodesic $Q_0$ splits the vertex set of a convex polyhedron into two subsets, the total curvature of each being $2 \pi$.
Alexandrov~\cite[pp.~377-378]{a-digdk-55} 
observed that such a condition is 
uncommon among all convex polyhedra.
This fact was further refined by 
Gruber~\cite{g-tcscc-91}
(as a preliminary step of his general result proof)
and by Gal'perin~\cite{gal2003convex} 
(for polyhedra homeomorphic to the sphere).
This led to a proof that,
for a fixed number of vertices, 
the set of convex polyhedra having a simple closed geodesic is closed and
has measure zero in the space of all convex polyhedra.

Particularizing to our framework, there is a special class of tetrahedra which do have many simple closed geodesics.
An \emph{isosceles tetrahedron}\footnote{Also called an \emph{isotetrahedron}, a \emph{tetramonohedron}, 
or an \emph{isohedral tetrahedron}.}
is a tetrahedron whose four vertices each have curvature $\pi$, 
or, equivalently, all four faces are congruent acute triangles.

It is a beautiful result that \emph{isosceles tetrahedra are the only convex surfaces that have arbitrarily long 
simple closed geodesics}
\cite{protasov2007closed}, \cite{akopyan2018long}.
Consequently, they have infinitely
many such geodesics.
This wealth of $Q_0$'s is balanced 
in some sense by the non-existence of $Q_1$'s.

\begin{lem}
\lemlab{isosceles}
No isosceles tetrahedron has a $1$-vertex quasigeodesic.
\end{lem}

This lemma complements the remark in~\cite{davis2017geodesics}, that a regular tetrahedron has no geodesic loop.
See also~\cite{strantzen1992regular}
for a characterization of isosceles tetrahedra as the only tetrahedra having
three distinct minimal loops through any point on the face.

\begin{proof}
See Fig.~\figref{TetraNo1VertQuasi}.
Here we use the second form of Gauss-Bonnet.
Let $Q$ be a $1$-vertex quasigeodesic through vertex $d$,
with $a$ and $b$ strictly to $Q$'s left, and $c$ strictly to $Q$'s right.
Since $\o_a+\o_b=2\pi$, $Q$ must have no turn, $\t=0$, to its left at $d$,
and turn $\t=\pi$ to its right at $d$.
Having no turn to its left means the total angle of $\pi$ is to the left of $Q$ at $d$.
Turning $\pi$ to the right means that $Q$ turns around completely,
folding back on itself, which then forces $Q$ to 
contain vertex $c$.
Thus $c$ does not lie strictly to $Q$'s right:
$Q$ is a 
$2$-vertex quasigeodesic,
not a $1$-vertex quasigeodesic.
See Fig.~\figref{TetraNo1VertQuasi} for an example.
\end{proof}
\begin{figure}[htbp]
\centering
\includegraphics[width=1.0\linewidth]{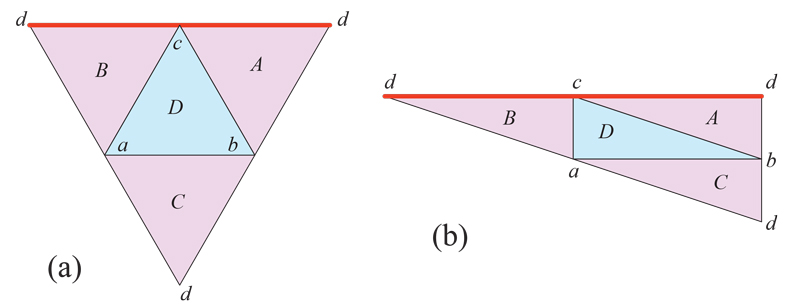}
\caption{Unfoldings of two isosceles tetrahedra. Quasigeodesic $Q$ (red) is a degenerate
doubling of edge $dc$.}
\figlab{TetraNo1VertQuasi}
\end{figure}


\section{$Q_1$: $1$-Vertex Quasigeodesics}
\seclab{Q1}

We have just seen in Section~\secref{Q0}
that isosceles tetrahedra have
no $1$-vertex quasigeodesic,
but do have simple closed geodesics.
The goal of this section is to prove this theorem:

\begin{thm}
\thmlab{Q1}
Every non-isosceles tetrahedron has 
at least one $1$-vertex simple closed quasigeodesic.
\end{thm}

In the remainder of this section, we assume all tetrahedra are not isosceles.


\paragraph{Properties of $Q_1$.}
A quasigeodesic $Q_1$ through exactly one vertex $v$
on a tetrahedron $T$ must satisfy these conditions.
\begin{enumerate}[label={(\arabic*)}]
\item $Q$ must form a geodesic loop with loop point $v$.
A \emph{geodesic loop} is a simple closed 
curve which is a geodesic everywhere but at $v$.
\item To satisfy the Gauss-Bonnet theorem, $Q$ must partition the other three vertices two to one side and one to the other,
such that the total curvatures to each side of $Q$ are at most $2\pi$
(and not both sides equal to $2\pi$, for then there is no curvature at $v$ and it is not a vertex).
%
\end{enumerate}
And of course the quasigeodesic angle criterion must be satisfied at $v$.


\paragraph{Sketch of Proof for Theorem~\thmref{Q1}.}

The proof follows a case analysis based first on how many curvatures are greater
than $\pi$, and second on the distances from low-curvature vertices to
high-curvature vertices.
The curvatures greater than $\pi$ lead to convex vertices in unfoldings,
and which vertices are closest to these high-curvature vertices permits
concluding that particular geodesic loops are inside certain disks and so live on $T$.
Then the angles to either side at the geodesic loop vertex must be verified to be at most $\pi$
to conclude it is a quasigeodesic.


\subsection{Case~1}
For Case~1, assume exactly one vertex has
curvature exceeding $\pi$: $\o_a > \pi$.
Let $d$ be the closest vertex to
$a$ among $b,c,d$.
Then 
star-unfold $T$ 
with respect to $d$, as illustrated in 
Fig.~\figref{Q1_Case1ab}:
Faces $C,D,B$ are incident to $a$,
and face $A$ is attached to face $D$ along edge $bc$.
Label the three images of $d$ as $d_1,d_2,d_3$ as illustrated.

We claim that $Q=d_1 d_2$ (red in the figure) is a simple closed quasigeodesic containing
just the vertex $d$. It will help to view $Q$ as directed from $d_1$ to $d_2$.
Note that, because $\o_a > \pi$, $\q_a < \pi$.

First note that, because $|ad|$ is shorter or equal to $|ab|$ and $|ac|$,
$Q$ separates $b,c$ from $a$:
$Q$ is a chord of a circle of radius $|ad|$ centered on $a$, and
$b$ and $c$ lie on or outside that circle.

Next, $Q$ is a straight segment between two images
of vertex $d$ in this unfolding, and so a geodesic loop on $T$ including $d$.
It remains to show that the angle to either side of $Q$ at $d$ is $\le \pi$.

Let $\a_1$ and $\a_2$ be the angles of $\triangle d_1 d_2 a$ above $Q$, as illustrated
in the figures.
Then it is immediate that $\a_1 + \a_2 < \pi$.

Let $\b_1$ and $\b_2$ be the angles
$\angle d_2 d_1 b$ and $\angle d_1 d_2 c$ below $Q$, as illustrated.
The angle of $Q$ to the right side of $d$ we seek to bound is $\b_1 + \b_2 + dA$.
The reason angle $dA$ is to the right of $Q$ is
(a)~$dA$ is incident to vertex $d$, and
(b)~it is not part of $\a_1+\a_2$ to the left of $Q$.

Now note that the external angles at $b$ and $c$ in the unfolding are
$\o_b$ and $\o_c$ respectively.
Because $\o_b,\o_c < \pi$, the triangle $\triangle d_1 d_2 d_3$
includes face $A$ and so includes $b$ and $c$.
Therefore $\b_1+\b_2+dA$ must be smaller than $\pi$
because those three angles are each smaller than the 
corresponding angles of  $\triangle d_1 d_2 d_3$.

Therefore we have proved that the angle to the right of $Q$ at $a$ is less than $\pi$,
and so $Q$ is a simple closed quasigeodesic as claimed.

\begin{figure}[htbp]
\centering
\includegraphics[width=1.0\linewidth]{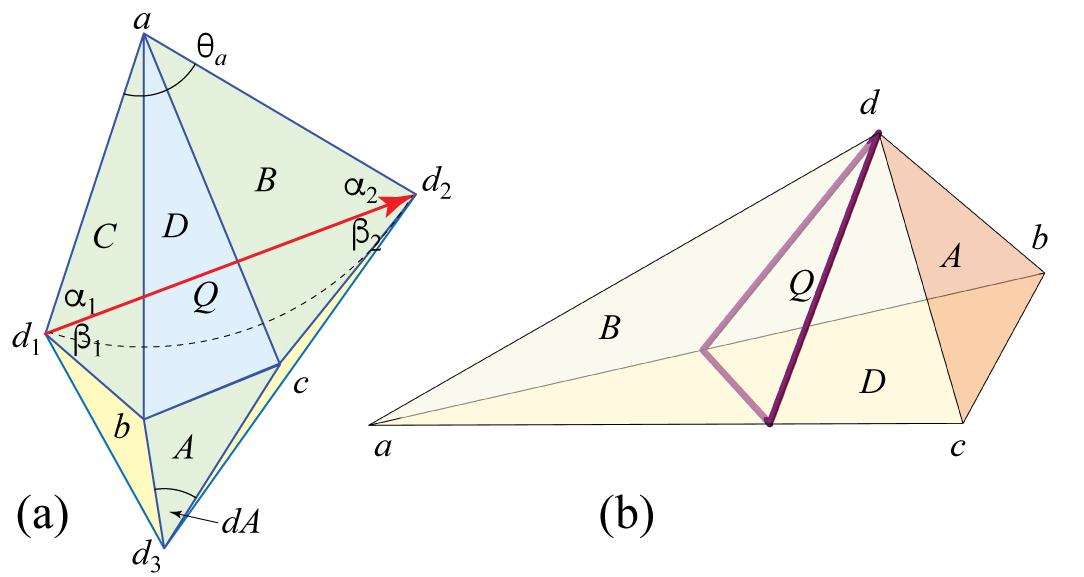}
\caption{Case~1. (a)~Unfolding of a tetrahedron
when  $\o_a > \pi$ and the other three curvatures are $< \pi$.
$d$ is closest to $a$.
Curvatures at $a,b,c,d$ are
$282^\circ,140^\circ,173^\circ,124^\circ$.
Auxiliary yellow triangles added.
(b)~The quasigeodesic in 3D.
}
\figlab{Q1_Case1ab}
\end{figure}


\newpage
\subsection{Case~2}
For Case~2,
assume $T$ has 
at least
two vertices with curvatures more than $\pi$:
$\o_a \ge \o_b \ge \pi$.

%

The reader may
find it easier to follow the proof on the particular case of
$T$ a doubly covered quadrilateral.
The next argument, however, is valid for the general situation.


We will consider geodesic loops $Q_v$ at vertex $v$, with $v \in \{a,b,c,d\}$ suitably chosen, as constructed at Case 1.

Consider a closest vertex $v$ to $a$ among $b, c, d$.

\paragraph{Case~2.1:} $v=b$: $|ab| \le |ac|,|ad|$: $b$ is closest to $a$.
Then there exists a geodesic loop $Q_b$ at $b$ 
which separates $a$ from $c, d$.
We now justify this claim.
As illustrated in Fig.~\figref{Case21}, 
the segment $Q_b = b_1 b_2$ cannot be blocked by vertex $a$
because $\q_a  \leq \pi$, 
and cannot be blocked by vertices $c$ or $d$,
because they fall outside the circle centered at $a$ of radius $|ab|$.

An equivalent but different way to view this is as follows.
View $a$ as the apex of $T$, and remove the base $bcd$. Extend the
faces $B,C,D$ incident to $a$ to a cone $\cal{C}$.
Then on $\cal{C}$
there are geodesic loops based on each of $b,c,d$.
Because $b$ is closest to $a$, the loop $Q_b$ lives on $T$.

Because $\o_b  \geq \pi$, the complete angle $\q_b$ at $b$ is at most $\pi$, 
hence $Q_b$ has less than $\pi$ to each side,
and so is a simple closed quasigeodesic.
This happens irrespective of $\o_c, \o_d$ being larger or smaller than $\pi$.

\begin{figure}[htbp]
\centering
\includegraphics[width=0.6\linewidth]{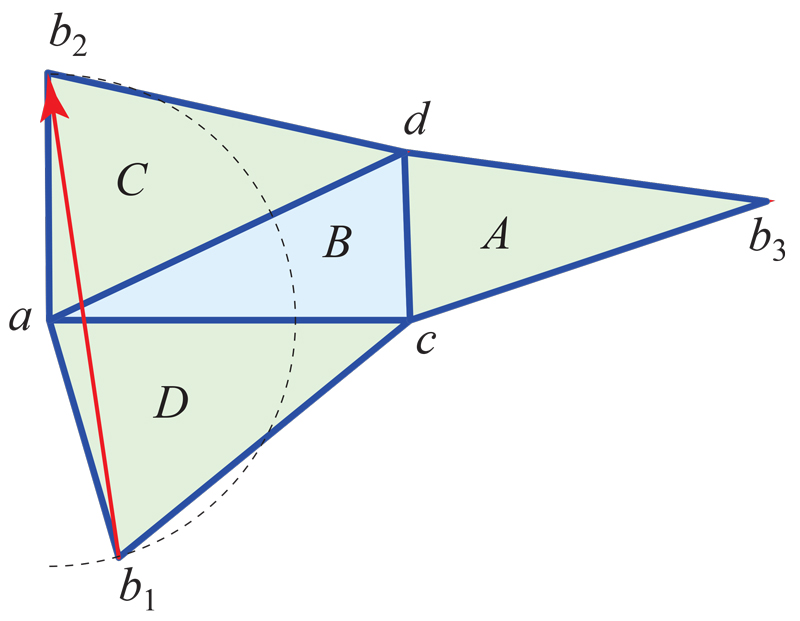}
\caption{Case~2.1: $b$ is closest to $a$. 
Curvatures at $a,b,c,d$ are
$196^\circ,190^\circ,159^\circ,175^\circ$.}
\figlab{Case21}
\end{figure}


\newpage
\paragraph{Case~2.2:} $v=c$; i.e., $c$ (or equivalently $d$) is closest to $a$: $|ac| \le |ab|,|ad|$.

Consider now a closest vertex $w$ to $b$ among $a, c, d$.

\paragraph{Case~2.2.1:} $w=a$: $|ba| \le |bc|,|bd|$.
This is handled by Case~2.1 with the roles of $a$ and $b$ reversed.

\paragraph{Case~2.2.2:}  $w=d$:
$|bd| \le |ba|,|bc|$. And $|ac| \le |ab|,|ad|$:
$c$ is closest to $a$ and $d$ is closest to $b$.

As illustrated in Fig.~\figref{Case222_cd}(ab),
the geodesic loop $Q_c$ at $c$ separates $a$ from $b,d$ 
because $c$ is closest to $a$ and so $b,d$ cannot interfere,
and the geodesic loop $Q_d$ at $d$ separates $b$ from $a,c$,
because $d$ is closest to $b$ and so $a,c$ cannot interfere.

Now we argue that one or the other of $Q_c, Q_d$ is a quasigeodesic.
Note first that, because $\o_a,\o_b \ge \pi$, the angle at $c$ toward $a$
(left in the figure), and the angle at $d$ toward $b$ (right in the figure), are $< \pi$
in $\triangle a c_1 c_2$ and $\triangle b d_1 d_2$ respectively.
Next we examine the angles to the other side of $Q_c, Q_d$.

With $a,b$ separated by $Q_c$ and $Q_d$, those two geodesic loops
bound a vertex-free region $R$ on the surface of $T$, homeomorphic to a cylinder.
As Fig.~\figref{Case222_cd}(c) illustrates, 
cut $R$ is isometric to a planar quadrilateral
$c_1 c_2 d_1 d_2$. Because the quadrilateral angles sum to $2\pi$, it cannot
be that the two angles at $c$ and the two angles at $d$ both exceed $\pi$.
At least one must be $\le \pi$.
The respective geodesic loop is therefore a quasigeodesic.
Fig.~\figref{Case222_3D} shows the two geodesic loops in 3D.

\begin{figure}[htbp]
\centering
\includegraphics[width=0.8\linewidth]{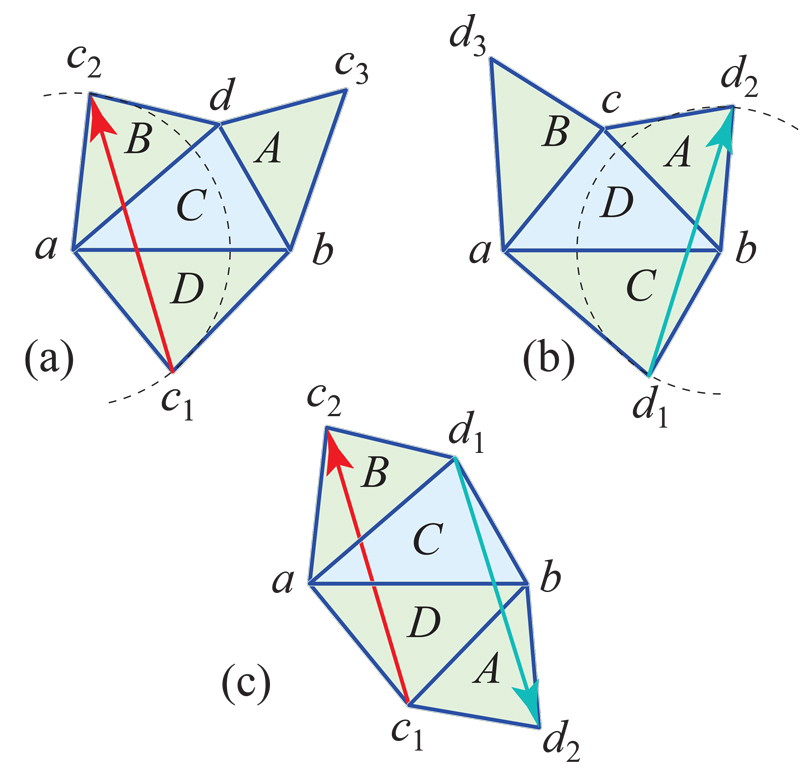}
\caption{Case~2.2.2: $c$ is closest to $a$ and $d$ is closest to $b$.
$Q_c$ red, $Q_d$ green.
Curvatures at $a,b,c,d$ are
$226^\circ,205^\circ,138^\circ,151^\circ$.
(The two segments $c_1 c_2$ and $d_1 d_2$ are slightly non-parallel.)
}
\figlab{Case222_cd}
\end{figure}
\begin{figure}[htbp]
\centering
\includegraphics[width=0.6\linewidth]{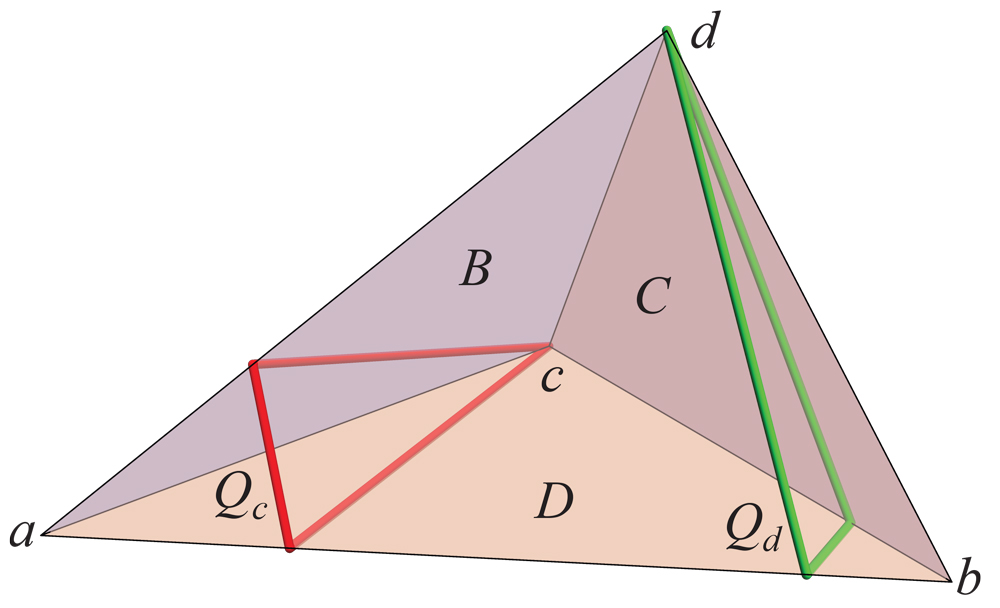}
\caption{Case~2.2.2: $Q_c$ and $Q_d$ on the 3D tetrahedron
unfolded in Fig.~\protect\figref{Case222_cd}. $Q_c$ is a quasigeodesic.}
\figlab{Case222_3D}
\end{figure}

%

\paragraph{Case~2.2.3:}   $w=c$; i.e., $c$ is closer than $d$ to both $a$ and $b$: 
$|bc| \le |ba|,|bd|$. And $|ac| \le |ab|,|ad|$:

As illustrated in Fig.~\figref{Case223_cd}(a),
just as in Case~1, 
because $c$ closest to both $a$ and $b$,
there exist geodesic loops $Q_c,Q'_c$ at $c$ such that
$Q_c$ separates $a$ from $b,d$ and 
$Q'_c$ separates $b$ from $a,d$:
vertex $d$ cannot interfere with either $Q_c=c_1 c_2$ nor $Q'_c=c_1 c_3$.

However, although one can construct two geodesic loops at $d$,
$Q_d$ on the cone apexed at $a$ and 
$Q'_d$ on the cone apexed at $b$,
they may not stay inside $T$.
We now argue that at least one of $Q_d$, $Q'_d$ lives on $T$.

Fig.~\figref{Case223_cd}(b) illustrates the situation when one, in this
case $Q_d=d_1 d_2$, falls outside $T$.
It should be clear that $d_1$ can see $c$, because $\o_a,\o_b \ge \pi$, so
the two faces $C \cup D$ form a convex quadrilateral.
The exterior angle gap at $c$ can only block visibility from one of $d_2,d_3$.
Another way to view this is that, if $Q_d=d_1 d_2$ does not live on $T$,
then it separates $a,c$ from $b$.
But then  $Q'_d=d_1 d_3$ even more so separates $a,c$ from $b$,
and lives on (remains inside the surface of) $T$.

\begin{figure}[htbp]
\centering
\includegraphics[width=0.95\linewidth]{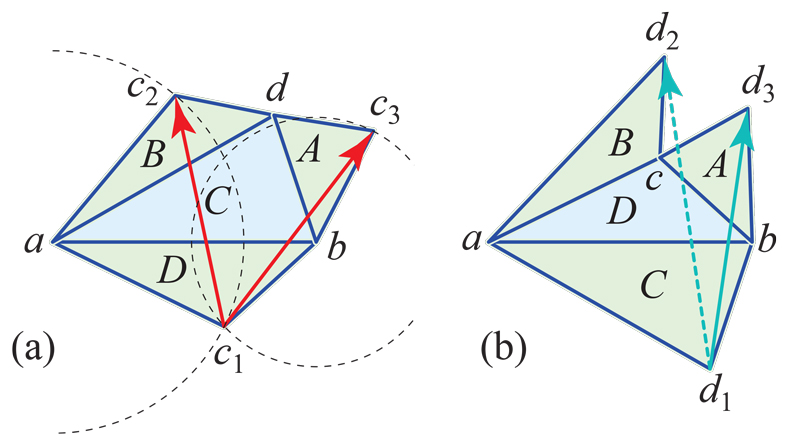}
\caption{Case~2.2.3: $c$ is closer than $d$ to both $a$ and $b$.
$Q_c,Q'_c$ red, $Q_d,Q'_d$ green.
Curvatures at $a,b,c,d$ are
$284^\circ,200^\circ,58^\circ,178^\circ$.
}
\figlab{Case223_cd}
\end{figure}

So now we have two geodesic loops, say, $Q_c = c_1 c_2$ and $Q_d= d_1 d_2$.
Then, just as in Case~2.2.2, we have identified a vertex-free region $R$, with
one geodesic excluding $a$ to the left, the other $b$ to the right.
And following the same logic as in Case~2.2.2, we conclude that at least one of
the angles toward $R$ at $c$ or $d$ must be $\le \pi$.
It is straightforward that the angles to the other side are $\le \pi$: 
 $\triangle a c_1 c_2$ and $\triangle b c_1 c_2$ and $\triangle b d_1 d_3$.

\newpage

This completes the proof of Theorem~\thmref{Q1}.
There is a sense in which this theorem cannot be strengthened, because there
are tetrahedra that have only one such $Q_1$.
(That more complex ``spiraling'' geodesic loops are
not possible is a consequence of~\cite[Lem.~8]{ChartierArnaud}.)
We claim that
Fig.~\figref{PointedOnly1Q1_2D} is an example of 
a tetrahedron with only one simple $Q_1$.

\begin{figure}[htbp]
\centering
\includegraphics[width=1.0\linewidth]{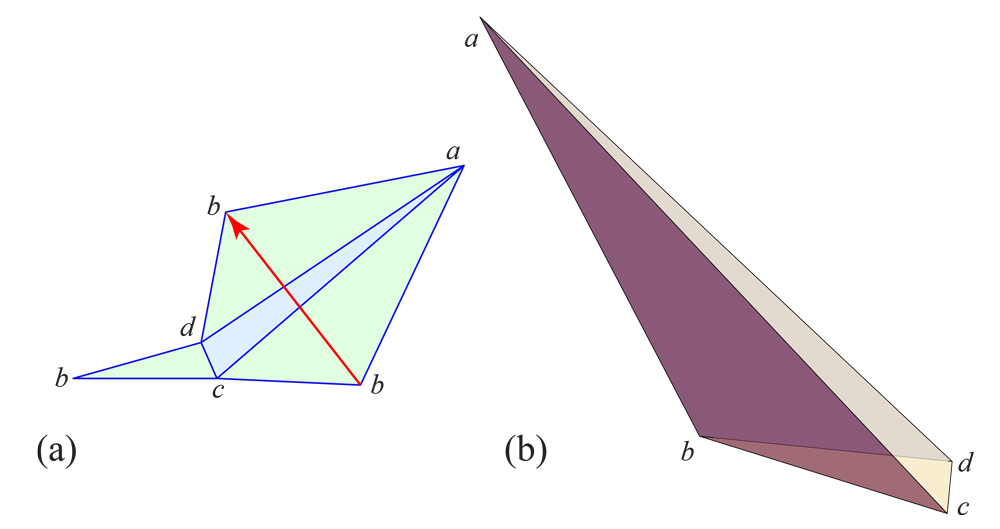}
\caption{Tetrahedron with just one $Q_1$.
Coordinates of $a,b,c,d$:
$(-0.65,0,1.56),(0,0,0),(1,0,0),(0.89,0.25,0)$.}
\figlab{PointedOnly1Q1_2D}
\end{figure}


\section{$Q_2$: $2$-Vertex Quasigeodesics}
\seclab{Q2}
The goal of this section is to prove this theorem:

\begin{thm}
\thmlab{Q2}
Every tetrahedron has a $2$-vertex simple closed quasigeodesic.
\end{thm}

\subsection{Degenerate $2$-vertex quasigeodesics}
If a tetrahedron has at least two vertices with curvature at least $\pi$,
then the complete angle incident to those two vertices is each $\le \pi$.
So the doubled edge connecting them constitutes a degenerate simple closed quasigeodesic:
at each endpoint, the angle to one side is $0$, and to the other side at most $\pi$.

Define a tetrahedron as \emph{pointed} if it has just one
vertex with curvature exceeding $\pi$.
We will consistently use the label $a$ for that vertex,
so it is \emph{pointed at $a$}.

The remainder of this section concentrates on pointed tetrahedra.
First we review some tools used in the proof.


\subsection{Star-Unfolding and Cut Locus}

A {\em geodesic segment} on a $P$ is a shortest path between its extremities.

The {\em cut locus} $C(x)$ of the point $x$ on $P$ is the set of endpoints (different from $x$) of all nonextendable geodesic segments (on the surface $P$) starting at $x$.

The \emph{star-unfolding} of $T$ with respect to its vertex $v$ is obtained by cutting along the edges incident to $v$ and 
unfolding to the plane.

We need one property of the star-unfolding that derives from~\cite{ao-nsu-92} and
is stated as Lemma~3.3 in~\cite{aaos-supa-97}.
To avoid introducing notation not needed here, we specialize this lemma
to our situation:

\begin{lem}
\lemlab{CutLocusTetra}
Let $S_v$ be the star-unfolding of a tetrahedron from vertex $v$.
Then the cut locus
$C(v)$ is the restriction to $S_v$ of the Voronoi diagram of the images of $v$.
Moreover, $C(v)$ lies entirely in the face opposite $v$.
In particular, the degree-$3$ ramification point $y$ lies in that face,
and is connected by segments to that face's three vertices.
\end{lem}

\noindent
One can see this intuitively: If $y$ were interior to a face incident to $v$, then there
would be three paths from $v$ to $y$: one straight in 3D, 
but the other two with some 3D aspect, a contradiction.

\paragraph{Sketch of Proof for Theorem~\thmref{Q2}.}
The proof first establishes a visibility relation in the star-unfolding that yields
an \emph{edge-loop}: a geodesic connecting to the endpoints of an edge of $T$.
Second, the quasigeodesic condition is proved to hold at both ends of the geodesic,
thereby establishing a $2$-vertex simple closed quasigeodesic.

\begin{figure}[htbp]
\centering
\includegraphics[width=0.4\linewidth]{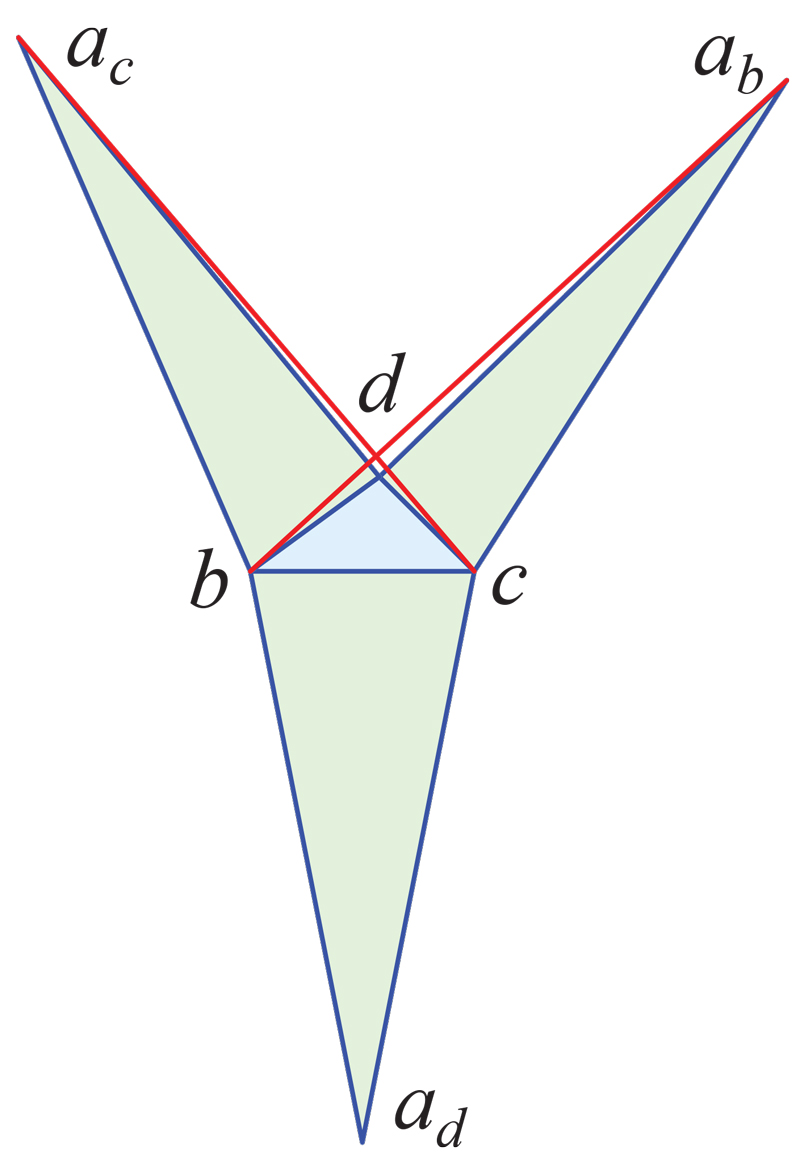}
\caption{Pointed $T$ with just one visible segment, $a_d d$. Both $a_b$ and $a_c$ are blocked.}
\figlab{PointedVisib_cex}
\end{figure}

\subsection{Visibility}
Let $T$ be a tetrahedron pointed at $a$, and
$S_a$ the star-unfolding of $T$ with respect to $a$.
Label the three images of $a$ as $a_b,a_c,a_d$, with $a_b$ opposite $b$,
and similarly for $a_c$ and $a_d$.
Say that two vertices $v,u \in S_a$ are \emph{visible} to one another if
the segment $vu$ is nowhere exterior to $S_a$,
and touches $\partial S_a$ only at $u$ and $v$.
So $uv$ is a vertex-to-vertex \emph{diagonal} of $S_a$.

\begin{lem}
\lemlab{Visibility}
If $T$ is pointed at $a$, then at least one of $b,c,d$ is visible in $S_a$
to $a_b,a_c,a_d$ respectively.
And sometimes there is only one such visibility relation.
\end{lem}
\begin{proof}
The tightness claim of the lemma is established by Fig.~\figref{PointedVisib_cex}.

Because $T$ is pointed at $a$, $S_a$ is reflex at $b,c,d$.
Partition the plane into six regions by extending the three edges of 
$\triangle bcd$. Call the cone regions $C_b,C_c,C_d$ incident to $b,c,d$ respectively.
If $a_b \in C_d$ or $a_b \in C_c$, then $a_b$ cannot see $b$, and similarly for $c$ and $d$.
So, for contradiction, we will show that if we have two visibility segments
$a_b b, a_c c$
blocked, then $a_d$ can see $d$.

Let $H(x,y)$ be the halfplane to the left of the directed segment $xy$.

\paragraph{Case 1: Two $a$-images in one cone.}
Let $a_b,a_c \in C_d$, with no loss of generality.
To contradict the claim of the lemma, we would need $a_d$ in either $C_c$ or $C_b$.
Again with no loss of generality, let $a_d \in C_c$.
This requires $a_d \in H(d,c)$, for the boundary of $H(d,c)$ is the (lower) boundary of $C_c$.
See Fig.~\figref{PointedVisib_Case1}.
At the same time, it must be that $a_d \in H(b,a_c)$ in order for $S_a$ to be reflex at $b$.
But it is not possible for $a_d$ to be in both of those halfplanes below $bc$,
for the intersection is only non-empty above $cd$.
\begin{figure}[htbp]
\centering
\includegraphics[width=0.9\linewidth]{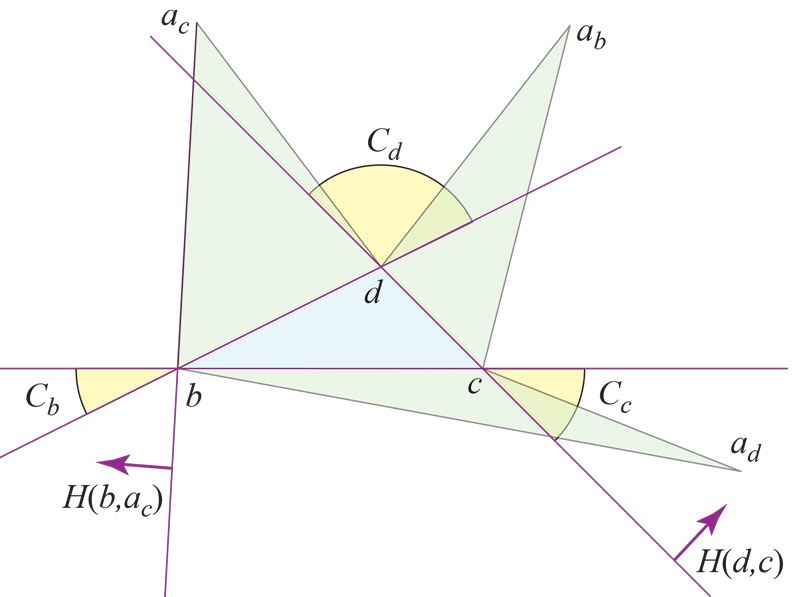}
\caption{Case~1. $a_d$ cannot be in both $C_c$ and $H(b,a_c)$.}
\figlab{PointedVisib_Case1}
\end{figure}

\paragraph{Case 2: Two $a$-images in two different cones.}
Let $a_b \in C_c$ and $a_c \in C_d$, with no loss of generality.
Then we seek to show that $a_d$ can see $d$.
If $a_d \in C_c$ or $a_d \in C_b$, then it is blocked from seeing $d$.
If $a_d \in C_c$, then both $a_d$ and $a_b$ are in the same cone, already
handled by Case~1. So we must have $a_d \in C_b$,
which requires $a_d \in H(b,d)$, as the boundary of $H(b,d)$ is the lower boundary of $C_b$.
See Fig.~\figref{PointedVisib_Case2}.

By Lemma~\lemref{CutLocusTetra}, the cut locus is entirely inside $\triangle bcd$.
Thus, each of $b,c,d$ connects by a segment to the ramification point $y$.
Each of these segments of the cut locus is a subsegment of a bisector.
In particular, the bisector of $a_b$ and $a_d$ lies along the segment $cy$.

Let $\b,\g,\d$ be the angles of $\triangle bcd$.
Now, in order for the bisector to penetrate into $\triangle bcd$ at $c$,
$a_d$ must lie in the halfplane $H(c,x)$, where the ray $cx$ makes an angle of $\g$ with
respect to the lower boundary of cone $Cc$.
In the example shown in Fig.~\figref{PointedVisib_Case2}, it is
clear that it is impossible for $a_d$ to be in both $H(c,x)$ and in $H(b,d)$,
because those halfplanes only intersect above $bc$.

However, if $\g$ is larger, then the two halfplanes could intersect below $bc$,
and thus there is a possible placement of $a_d$ in $C_b$ satisfying the bisector
condition. It is easy to see that the critical inequality is that we need $2 \g > \pi - \b$
for a placement to be available.

But now a similar inequality is needed for the placements of $a_c$ and $a_b$, for
each of those to both lie in the appropriate cones, and lead to bisectors that penetrate
into  $\triangle bcd$ at $b$ and $d$ respectively.
Thus these three inequalities must hold:
\begin{align*}
2 \g &> \pi - \b \\
2 \b &> \pi - \d \\
2 \d &> \pi - \g 
\end{align*}
The reason that the inequalities are strict is that equality implies that
the boundaries of $H(b,d)$
and $H(c,x)$ are parallel, and there is no spot at which to locate $a_d$ (and similarly
for $a_c$ and $a_b$).
Summing the three inequalities leads to $2 \pi > 2 \pi$, a contradiction.

Therefore, it is not possible to have all three of $a_b, a_c, a_d$ located
in three different cones. 

Since the two Cases cover all possibilities,  
at least one image of $a$ must lie in a region between
cones, and so can see the corresponding vertex of $\triangle bcd$.
\end{proof}
\begin{figure}[htbp]
\centering
\includegraphics[width=0.9\linewidth]{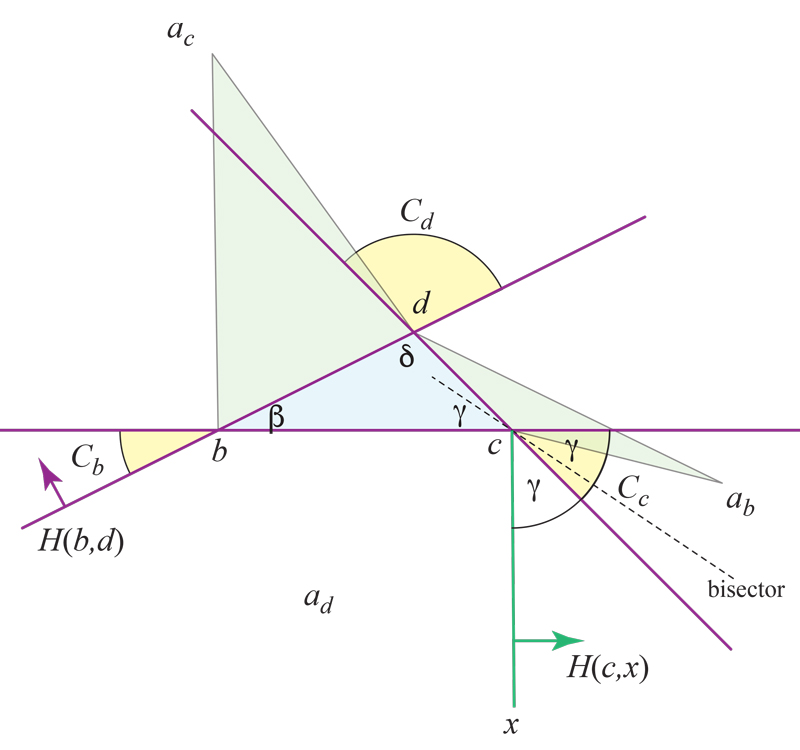}
\caption{Case~2: $a_d$ must lie in $H(c,x)$ for the bisector of
$a_d$ and $a_b$ to penetrate $\triangle bcd$ at $c$.}
\figlab{PointedVisib_Case2}
\end{figure}

As a consequence of Lemma~\lemref{Visibility}, we have established that every pointed
tetrahedron has an edge-loop. Our next goal is to show that this edge-loop
is in fact a simple closed quasigeodesic, which requires at most $\pi$
angle at the endpoints of the visibility segment.

\begin{lem}
\lemlab{LoopQuasi}
Every pointed tetrahedron has a non-degenerate edge-digon forming
a simple closed quasigeodesic.
\end{lem}

\begin{proof}
Let the visibility  segment be $a_d d$ without loss of generality.
First, because the curvature at $a$ is $> \pi$, the total angle there
is $< \pi$, and so the quasigeodesic condition is satisfied at that end.
Now we turn to the other end at $d$.

Consider the segments incident to $a_d$.
Segment $a_d b$ is left of and $a_d c$ is right of $a_d d$, just by the
counterclockwise labeling convention for the base $bcd$.
Now, because $\o_b < \pi$ and $\o_c < \pi$, the angles at $b$ and at $c$
in the unfolding $S_a$ are both reflex. 
See Fig.~\figref{AngSplit}.
Therefore the segment endpoints incident
to $a_d$ follow the counterclockwise order $a_b, c, d, b, a_c$.
Therefore the angle at $d$ right of $a_d d$ is the apex of the triangle 
$\triangle a_b d a_d$ and the angle to the left is the apex of $\triangle a_c d a_d$.
Because both angles are $< \pi$, we have established that
the edge-loop is in fact a simple closed quasigeodesic.
\end{proof}
\begin{figure}[htbp]
\centering
\includegraphics[width=0.7\linewidth]{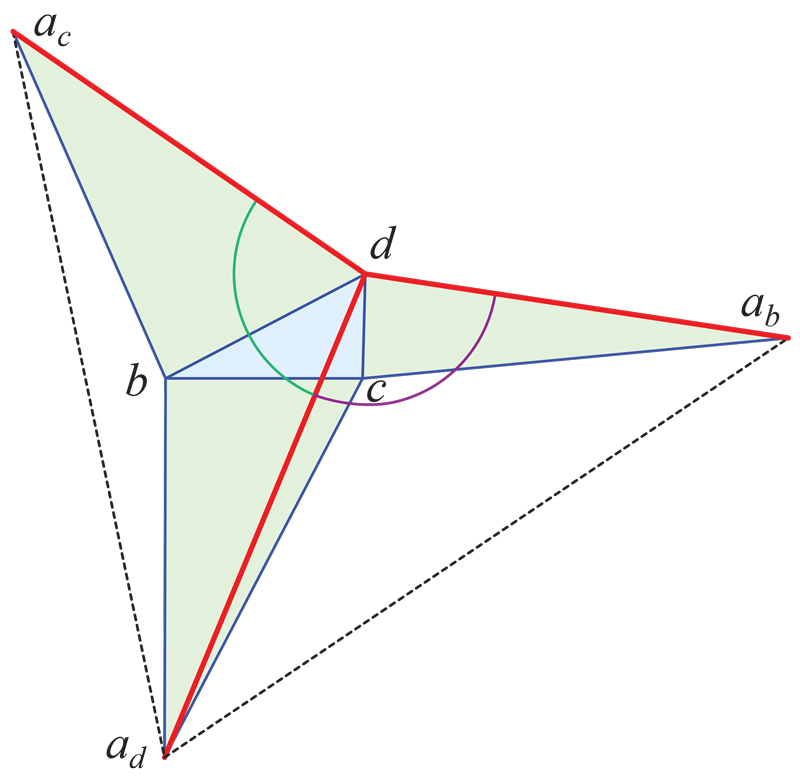}
\caption{$\triangle a_b d a_d$ and $\triangle a_c d a_d$:
triangle angles at $d$ are $<\pi$.}
\figlab{AngSplit}
\end{figure}

Note Lemma~\lemref{LoopQuasi} holds for every visibility segment: Although there is always at least one,
there can be as many as three.

Together with the degenerate $2$-vertex quasigeodesics on non-pointed tetrahedra,
we have established Theorem~\thmref{Q2}.


\subsection{A geometrical interpretation of edge loops}
In this subsection we provide a geometrical interpretation of edge loops of tetrahedra.
Our construction is based on Alexandrov's Gluing Theorem and the technique of vertex merging;
see~\cite{Reshaping} for a description and other applications of these tools.

We are still in the case of tetrahedra $T$ pointed at $a$, with an edge loop based on the edge $ad$, see Fig.~\figref{AngSplit}.
Then we have $\q_a + \q_d < 2 \pi$.

Denote by $n$ the mid-point of the edge $ad$.
Cut $T$ along $ad$ and glue back differently, via Alexandrov's Gluing Theorem.
Precisely, we identify $a$ and $d$, and for the two banks of the cut, $an$ with $dn$.
Because $\q_a + \q_d < 2 \pi$, the result after gluing is a convex polyhedron $F$ with $5$ vertices: 
$b$, $c$, $n_1$ and $n_2$ obtained from $n$, and $w=ab$ obtained from both $a$ and $b$.
By construction, the curvatures at $n_1, n_2$ are precisely $\pi$.

On $F$, we can merge the vertices $c$ and $n_1$, and respectively $d$ and $n_2$, to obtain a new convex polyhedron $\D$ with three vertices: 
$w$, $u$ obtained from  $c$ and $n_1$, and $v$ obtained from $d$ and $n_2$.
Therefore, $\D$ is a doubly covered (obtuse) triangle.
See Fig.~\figref{Q2Delta_bside}.

The edge loop of $T$ based on $ad$ corresponds to the geodesic loop at $w$ on $\D$ obtained by doubling the height $h$ from $w$.

\begin{figure}[htbp]
\centering
\includegraphics[width=1.0\linewidth]{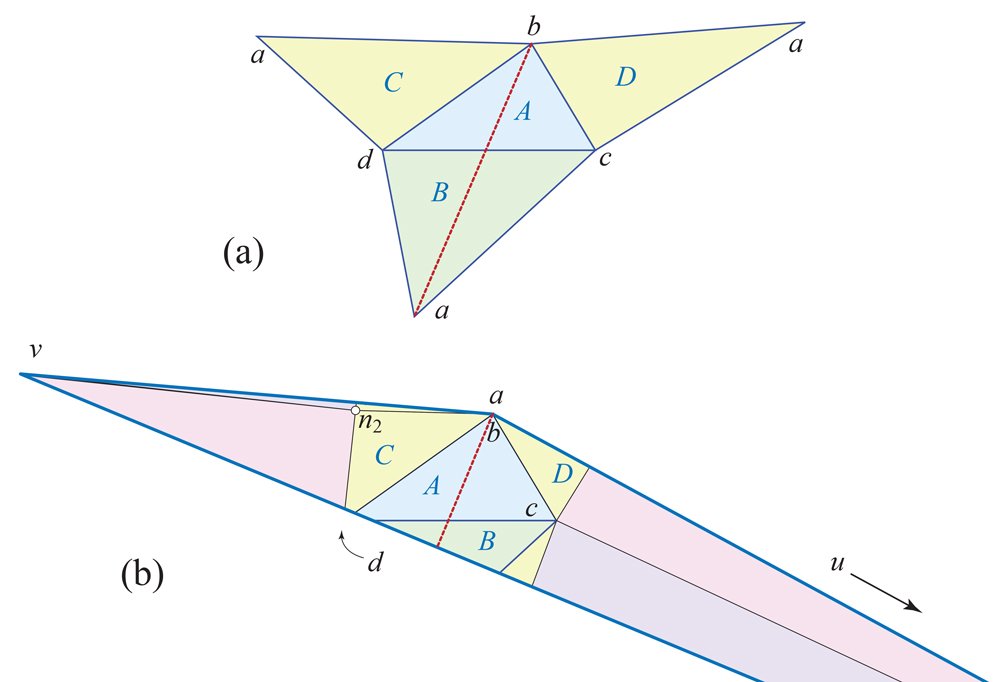}
\caption{(a)~Unfolding of a pointed tetrahedron, with edge-loop 
red.
(b)~One side of $\D$, an obtuse doubly-covered triangle with edge-loop along the altitude.
($n_1$ is on the back side.)}
\figlab{Q2Delta_bside}
\end{figure}


\newpage
\section{$Q_3$: $3$-vertex Quasigeodesics}
\seclab{Q3}

The goal of this section\footnote{
This section is a revision of~\cite{o2021every}.} 
is to prove this theorem:
\begin{thm}
\thmlab{Q3}
Every tetrahedron $T$ has at least one face $F$ whose boundary
$\partial F$ is a $3$-vertex simple closed quasigeodesic $Q_3$.
\end{thm}
\noindent
We first establish two preliminary lemmas that will be used in the proof.


\subsection{Preliminary Lemmas}

\begin{lem}
\lemlab{trineq}
Let $\a_1,\a_2,\a_3$ be the face angles incident to vertex $v$
of a tetrahedron $T$.
Then the angles satisfy the triangle inequality: 
$\a_1 < \a_2 + \a_3$, and similarly
$\a_2 < \a_1 + \a_3$, and $\a_3 < \a_1 + \a_2$.
The inequalities are strict unless $T$ is flat.
\end{lem}
\begin{proof}
See Lemma~2.8 in~\cite{Reshaping}.
\end{proof}

\bs 

To simplify the calculations, angles will be represented 
in inequalities in units of $\pi$:
$1 \equiv \pi$, $2 \equiv 2\pi$, etc.
Thus under this convention, each of the $12$ face angles of a tetrahedron lies in $(0,1)$.

We say that ``\emph{face $F$ fails at vertex $v$}" if the two angles incident to $v$ not in $F$
exceed $\pi$.
So, for face $A$ to fail on vertex $b$,
then among the three
face angles $bA,bC,bD$ incident to $b$, the two angles not in $A$
satisfy $bC + bD > 1$. 
This means that $\partial A$ is not a quasigeodesic, 
because to one side---the other side from $bA$---the angle exceeds $\pi$.

\paragraph{Example.}
Fig.~\figref{OneQ3} shows a tetrahedron with 
$\partial C$ a quasigeodesic, but none of
the other face boundaries is a quasigeodesic.
Thus the ``at least one" claim of Theorem~\thmref{Q3} cannot be strengthened.
Its vertex coordinates are:
\begin{equation*}
a,b,c,d =
 (-3.54,1.98,4.58) \; (0,0,0), \; (1,0,0), \;  (4.91,3.24,0) \;.
\end{equation*}
For example, back face $C$ does not fail at vertex $b$:
$bD+bA= 125^\circ + 33^\circ = 159^\circ < \pi$.
Front face $B$ fails at vertex $c$: $cD + cA = 48^\circ + 140^\circ = 188^\circ > \pi$.
\begin{figure}[htbp]
\centering
\includegraphics[width=0.8\linewidth]{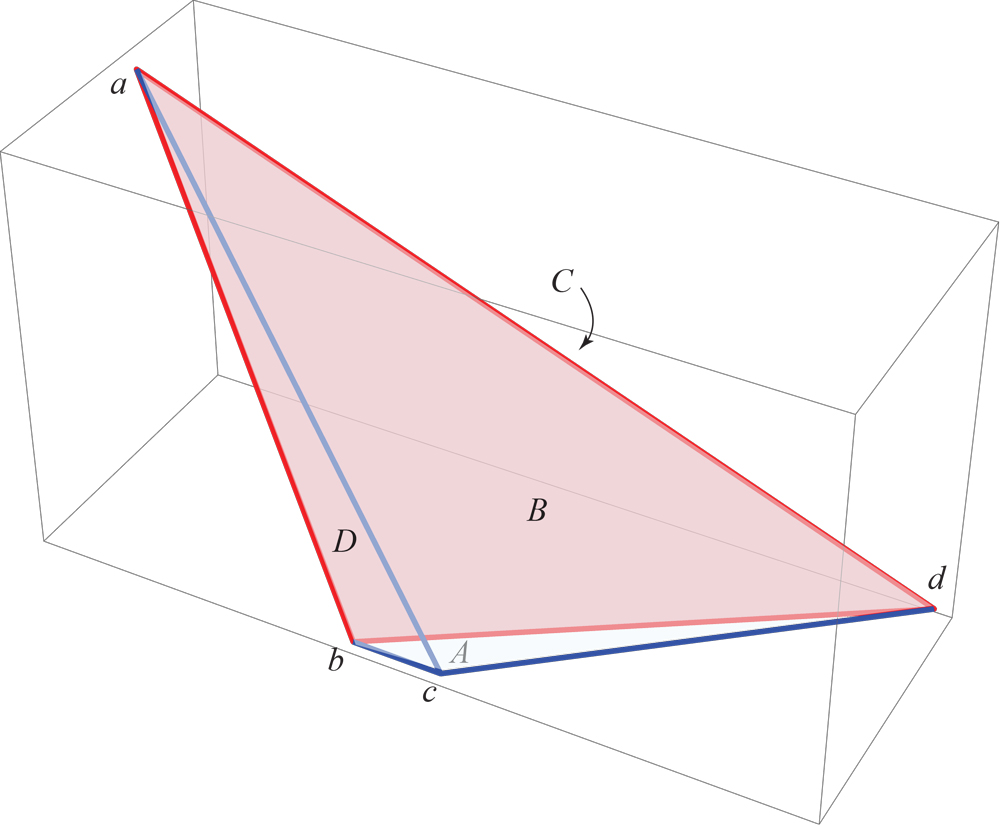}
\caption{The (red) boundary of shaded back face $C=abd$ is a quasigeodesic,
but none of $\partial A, \partial B, \partial D$ are quasigeodesics.}
\figlab{OneQ3}
\end{figure}

\begin{lem}
\lemlab{curvless1}
If a face $A$ fails at a vertex $b$, then $\omega(b)<1$.
\end{lem}
\begin{proof}
Since face $A$ fails at $b$, by definition, $bC + bD > 1$.
Therefore 
\begin{eqnarray*}
\omega(b) &=& 2-(bA+bC+bD)\\
\omega(b) &=& 2-(bC+bD) - bA\\
\omega(b) &<& 1 - bA\\
\omega(b) &<& 1
\end{eqnarray*}
This establishes the claim of the lemma.
\end{proof}


\subsection{Case Analysis}
We now undertake a case analysis to show that it is not possible for all
four faces of tetrahedron $T$ to fail at vertices.
The cases, illustrated in Fig.~\figref{Cases}, distinguish first the number
of distinct vertices among the four face-failures, and second, the pattern
of the failures.

\begin{figure}[htbp]
\centering
\includegraphics[width=0.7\textheight]{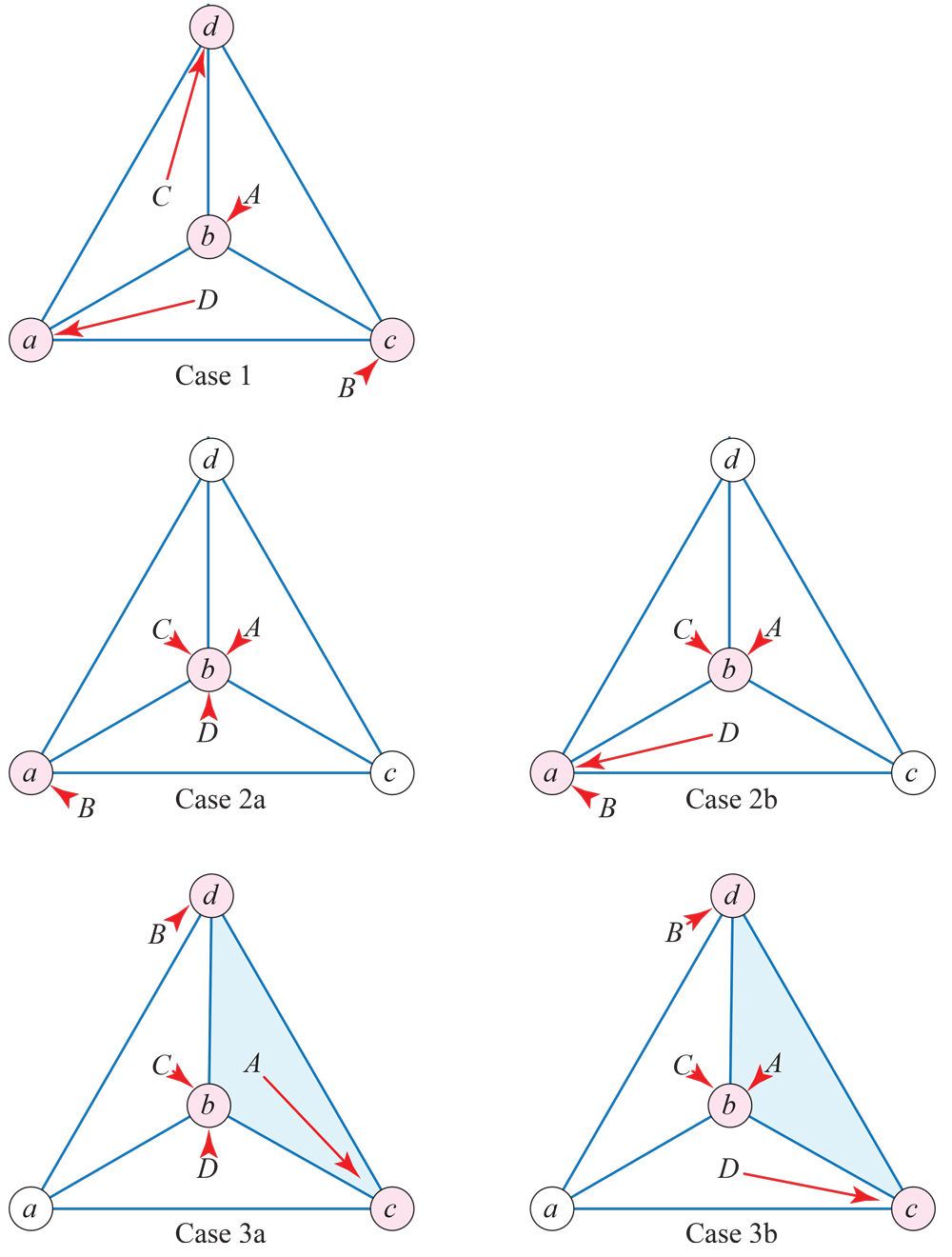}
\caption{Failures. Case~1: at $4$ vertices.
Case~2: at $2$ vertices. 
Case~3: at $3$ vertices.}
\figlab{Cases}
\end{figure}

The proof analyzes the $12$ face angles of $T$, and shows
the set of solutions in $(0,1)^{12}$ is empty
(under the convention that each angle is in $(0,1)$).
So we are representing tetrahedra by their $12$ face angles.
The four faces each have a total of $\pi$ angle, which
reduces the dimension of the tetrahedron configuration space from $12$ to $8$.
It is known that in fact the configuration space is $5$-dimensional, 
not $8$-dimensional~\cite{rouyer2012sets},
but the proof to follow works without including the various additional trigonometric relations
that tetrahedron angles must satisfy. It suffices to use linear equalities and inequalities
among the $12$ face angles. 

\paragraph{Case 1: $4$ vertices.}
Suppose first that each of the four faces $A,B,C,D$ fail on four
distinct vertices. Then
Lemma~\lemref{curvless1} shows that
$\omega(v) < 1$ for each vertex $v$.
But then $\sum \omega(v) < 4$, contradicting the Gauss-Bonnet theorem.

\paragraph{Case 2a: $2$ vertices, $3+1$.}
Suppose now that the four faces fail on a total of two vertices.
This can occur in two distinct ways: three faces fail on one vertex,
which we call Case~2a, or two faces fail each on two vertices, Case~2b.
Say that $b$ is the vertex at which three faces fail. We then have:

\begin{eqnarray*}
A \;\mathrm{fails} \;\mathrm{at}\; b &:& bC + bD > 1 \\
B \;\mathrm{fails} \;\mathrm{at}\; a &:& aC + aD > 1 \\
C \;\mathrm{fails} \;\mathrm{at}\; b \\
D \;\mathrm{fails} \;\mathrm{at}\; b
\end{eqnarray*}
It turns out that we do not need to use the fact that
$C$ and $D$ fail at some vertices, so the implied inequalities are suppressed.
Summing the failure inequalities above leads to a contradiction:
\begin{eqnarray*}
(aC + bC) + (aD + bD) &>& 2 \\
(1 - dC) + (1 - cD) &>& 2 \\
2 &>& 2 + (dC + cD) \\
0 &>& dC + cD
\end{eqnarray*}
This is a contradiction because all angles have positive measure.

\paragraph{Case 2b: $2$ vertices, $2+2$.}
This follows the exact same proof, as again $C$ and $D$ failures are not needed
to reach a contradiction.

\paragraph{Case 3a: $3$ vertices, double outside.}
The three vertices at which faces fail bound a face, say $A$.
One vertex of $A$, say $b$, is ``doubled" in the sense that two faces fail at $b$.
Case~3a is distinguished in that neither face failing on $b$ is the three-vertex face $A$.
(Swapping $B$ to fail on $c$ and $A$ to fail of $d$
is symmetrically equivalent to the case illustrated.)

We again do not need all failures, in particular, we only need
those for faces $B$ and $D$:
\begin{eqnarray*}
A \;\mathrm{fails} \;\mathrm{at}\; c \\ 
B \;\mathrm{fails} \;\mathrm{at}\; d &:& dA + dC  > 1 \\
C \;\mathrm{fails} \;\mathrm{at}\; b  \\
D \;\mathrm{fails} \;\mathrm{at}\; b &:& bA + bC  > 1 \\
\end{eqnarray*}
\noindent
Adding these inequalities leads to the same contradiction:
\begin{eqnarray*}
(bA + dA) + (bC + dC) &>& 2 \\
(1 - cA) + (1 - aC) &>& 2 \\
0 &>& aC + cA
\end{eqnarray*}
Again a contradiction.

\paragraph{Case 3b: $3$ vertices, double inside.}
In contrast to Case~3a, in this case, one of the faces that fail on $b$ is
the three-vertex face $A$.
(Swapping $B$ to fail on $c$, $D$ to fail on $b$, and $C$ to fail on $d$,
is symmetrically equivalent.) 
This is the only difficult case, and the only case in which the
triangle inequalities guaranteed by Lemma~\lemref{trineq} are needed.

The angles of face $A$ satisfy $bA + cA + dA=1$.
Assume without loss of generality that
$bA \le cA \le dA$.
Three faces, $B,C,D$ fail at the three vertices of face $A$:
$d,b,c$ respectively.

To build intuition, we first run through the proof for specific $A$-face angles:

\begin{eqnarray*}
( bA, cA, dA ) &=& (0.1, 0.3, 0.6)\\
A \;\mathrm{fails} \;\mathrm{at}\; b \\ 
B \;\mathrm{fails} \;\mathrm{at}\; d &:& dA + dC > 1 \;:\; dC  > 0.4 \\
C \;\mathrm{fails} \;\mathrm{at}\; b &:& bA + bD > 1 \;:\; bD > 0.9 \\
D \;\mathrm{fails} \;\mathrm{at}\; c &:& cA + cB > 1 \;:\; cB  > 0.7
\end{eqnarray*}
Note $0.4 + 0.9 + 0.7 = 2$; this holds for arbitrary $A$ angles.
Now apply the triangle inequality to each of $dB, cB, dC$:
\begin{eqnarray*}
bD &<& bA + bC \;:\; bC > bD - bA \;:\; bC > 0.8 \\
cB &<& cA + cD \;:\; cD > cB -  cA \;:\;  cD > 0.4 \\
dC &<& dA + dB \;:\; dB > dC - dA \;:\;  dB > -0.2
\end{eqnarray*}
Note $0.8 + 0.4 -0.2 = 1$; this again holds for arbitrary $A$ angles.

Triangle face $D$ satisfies: $bD + cD + aD =1$.
\begin{eqnarray*}
bD &>& 0.9 \\
cD &>& 0.4 \\
bD + cD &>& 1.3 \\
bD + cD + aD &>& 1.3 \;>\; 1
\end{eqnarray*}
which contradicts $bD + cD + aD =1$.

\medskip
\begin{center}
\noindent\rule{0.5\textwidth}{0.5pt}
\end{center}
\medskip

Without specific angles assigned to $( bA, cA, dA )$, the argument
is less transparent.
Again assume that $bA \le cA \le dA$.

\begin{eqnarray*}
A \;\mathrm{fails} \;\mathrm{at}\; b \\ 
B \;\mathrm{fails} \;\mathrm{at}\; d &:& dA + dC > 1 \;:\; dC  > 1 - dA  \\
C \;\mathrm{fails} \;\mathrm{at}\; b &:& bA + bD > 1 \;:\; bD > 1 - bA \\
D \;\mathrm{fails} \;\mathrm{at}\; c &:& cA + cB > 1 \;:\; cB  > 1 - cA 
\end{eqnarray*}
Note the sum of the above three right-hand sides is $3 - (dA+bA+cA) = 2$.
\noindent
Now apply the triangle inequality to $dB, cB, dC$:
\begin{eqnarray*}
bD &<& bA + bC \;:\; bC > bD - bA \;:\; bC > 1 - 2 \cdot bA \\
cB &<& cA + cD \;:\; cD > cB -  cA \;:\;  cD > 1 - 2 \cdot cA \\
dC &<& dA + dB \;:\; dB > dC - dA \;:\;  dB > 1 - 2 \cdot dA
\end{eqnarray*}
Note the sum of the above three right-hand sides is $3 - 2(dA+bA+cA) = 1$.
\noindent
Face $D$'s angles satisfy  $bD + cD + aD =1$.
Now we reach a contradiction using the inequalities above.
\begin{eqnarray*}
bD &>& 1 - bA \\
cD &>& 1 - 2 \cdot cA \\
bD + cD &>& 2 - (bA + 2 \cdot cA )
\end{eqnarray*}
We have $(bA + 2 \cdot cA ) \le 1$ because $bA + cA + dA=1$
and $cA \le dA$.
And of course every angle is positive, so $aD > 0$.
So we have
\begin{eqnarray*}
bD + cD &>& 1\\
bD + cD + aD &>& 1
\end{eqnarray*}
which contradicts $bD + cD + aD =1$.

That the inequalities for each of the above cases
cannot be simultaneously satisfied
has been verified by Mathematica's
\texttt{FindInstance[]} function, which
uses Linear Programming over the rationals\footnote{
\url{https://mathematica.stackexchange.com/q/255494/194}}
to conclude that the set of 
solutions in $\mathbb{R}^{12}$ is empty.

Replacing the triangle inequalities with equalities when the tetrahedron
is flat (e.g., $aB = aC + aD$ instead of $aB < aC + aD$) again leads to the same contradiction.

We have thus established Theorem~\thmref{Q3}.
And together with remarks in Section~\secref{Q0},
Theorems~\thmref{Q1}, \thmref{Q2}, and \thmref{Q3}
establish Theorem~\thmref{Q123}.


\section{Tetrahedra with many $Q_{1,2,3}$}
\seclab{F-acute}

As mentioned in the Introduction (Section~\secref{Introduction}),
one cannot expect there to exist more than three simple closed quasigeodesics on a general convex surface,
so one could expect the same fact
also holds for general tetrahedra.

In this section we provide an open subset of the space of tetrahedra, 
each tetrahedron of which has (unexpectedly) many such quasigeodesics.

Let  ${\cal T}$ be the space of all tetrahedra in $\Rs$,
with the topology induced by the usual Pompeiu-Hausdorff metric.
Two polyhedra in ${\cal T}$ are then close to each other if and only if they have close respective vertices.

The goal of this section is to prove the 
theorem previously stated in the Introduction:

\setcounter{thm}{1} 
\begin{thm}
\thmlab{34again}
There exists an open set ${\cal O}$ of tetrahedra, each element of which has at least $34$ simple closed quasigeodesics.
\end{thm}

We call a tetrahedron \emph{f-acute} if all its faces are acute triangles.

\begin{lem}
\lemlab{f-acute_open}
The set ${\cal F}$ of f-acute tetrahedra is open in $\cal{T}$.
\end{lem}

\begin{proof}
The face angles at the vertices of $T$ depend continuously on the vertex positions in $\Rs$.
Once they are $<\pi$, they remain so in a neighborhood. 
\end{proof}

We further restrict our study to a special
open subset ${\cal O}$ of ${\cal F}$, of tetrahedra near a regular tetrahedron, all having three vertices of curvature $> \pi$.
The introduction of these tetrahedra is justified by the considerations in Section~\secref{Q2}.

We start with a regular tetrahedron of apex $a$ 
and horizontal base $bcd$
and move $a$ 
downward
a short distance along a vertical line.
The new tetrahedron $N$ has base vertices of curvatures slightly larger than $\pi$ 
and top vertex $a$ of curvature slightly less than $\pi$. 
Moreover, all faces of $N$ remain acute triangles.
So we consider ${\cal O}$ to be a small neighborhood of $N$.

The next lemma can be proved with an argument similar to 
Lemma~\lemref{f-acute_open}'s proof.

\begin{lem}
\lemlab{3_large_curvatures_open}
All tetrahedra in ${\cal O}$ are f-acute and have three vertices of curvature $> \pi$.
\end{lem}

The following is a particular case of Lemma 17.2 in~\cite{Reshaping}.

\begin{lem}
\lemlab{Q_on-neighborhood}
Assume the tetrahedron $T$ has a simple closed quasigeodesic $Q_k$ through $k \geq 1$  vertices, such that
its left and right angles at each of the $k$ vertices are all strictly less than $\pi$.
Then, all tetrahedra sufficiently close to $T$ in ${\cal T}$ have such a quasigeodesic.
\end{lem}

\begin{figure}[htbp]
\centering
\includegraphics[width=0.75\linewidth]{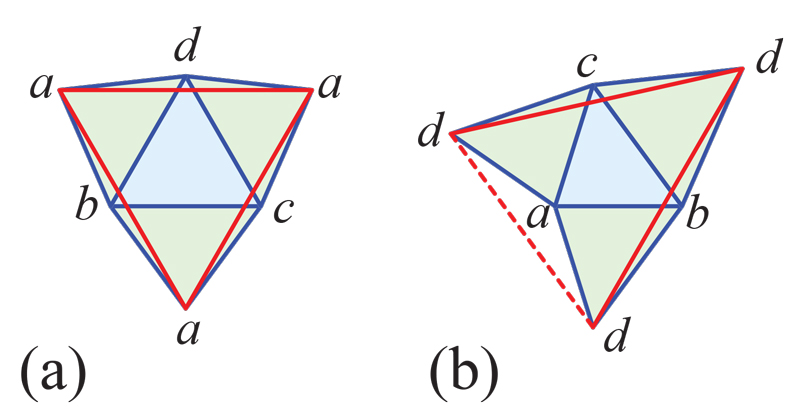}
\caption{Two star unfoldings of a near-regular tetrahedron.
In this example,
$\o_a=142^\circ$ and 
$\o_b=\o_c=\o_d=193^\circ$.
Geodesic loops are solid red segments. 
In~(a), each loop has angle $12^\circ$ to one side of $a$ and $206^\circ$ to the other side.
In~(b), $\q_d=167^\circ$,
so the loops are quasigeodesics.
}
\figlab{Prolific}
\end{figure}

In view of Lemmas \lemref{3_large_curvatures_open} and \lemref{Q_on-neighborhood},
it suffices to count the simple closed quasigeodesics on 
our reference tetrahedron $N \in {\cal O}$.
\begin{itemize}
\item We saw that a general tetrahedron has no $Q_0$.

\item There exists at least one $Q_1$ on every tetrahedron. In fact, $N$ has ${\bf 6}$ such quasigeodesics.
To see this, consider the four star-unfoldings of $N$ with respect to its vertices. Because of the symmetry of $N$, three of the unfoldings from
the base vertices $b,c,d$ are isometric.
See Fig.~\figref{Prolific}. One can then check that through each base vertex pass two $Q_1$'s, as represented in Fig.~\figref{Prolific}(b).
On the other hand, the three geodesic loops through apex $a$, represented in Fig.~\figref{Prolific}(a), are not quasigeodesics.

\item Because $N$ is chosen sufficiently close to a regular tetrahedron, 
Lemma~\lemref{Q_on-neighborhood} shows that every edge of $N$ provides three non-degenerate $Q_2$'s,
as in Fig.~\figref{Q2_alt_reg_3D}. The three vertices of curvatures $> \pi$ provide three more degenerate $Q_2$'s.
They all sum up to ${\bf 21}$ $Q_2$'s.

\item The boundary of every face of $N$ is a $Q_3$, because $N$ is f-acute, hence there are ${\bf 4}$ such $Q_3$ quasigeodesics.

\item Every partition of the face set of $N$ 
into two faces
provides a $Q_4$, again because $N$ is f-acute, hence there are ${\bf 3}$ $Q_4$'s, namely corresponding to 
$AB:CD$, $AC:BD$, $AD:BC$.
\end{itemize}

Thus we have found a tetrahedron $N$ in whose neighborhood $\cal{O}$,
every tetrahedron has at least $34$ quasigeodesics, verifying Theorem~\thmref{34}.

\begin{figure}[htbp]
\centering
\includegraphics[width=0.75\linewidth]{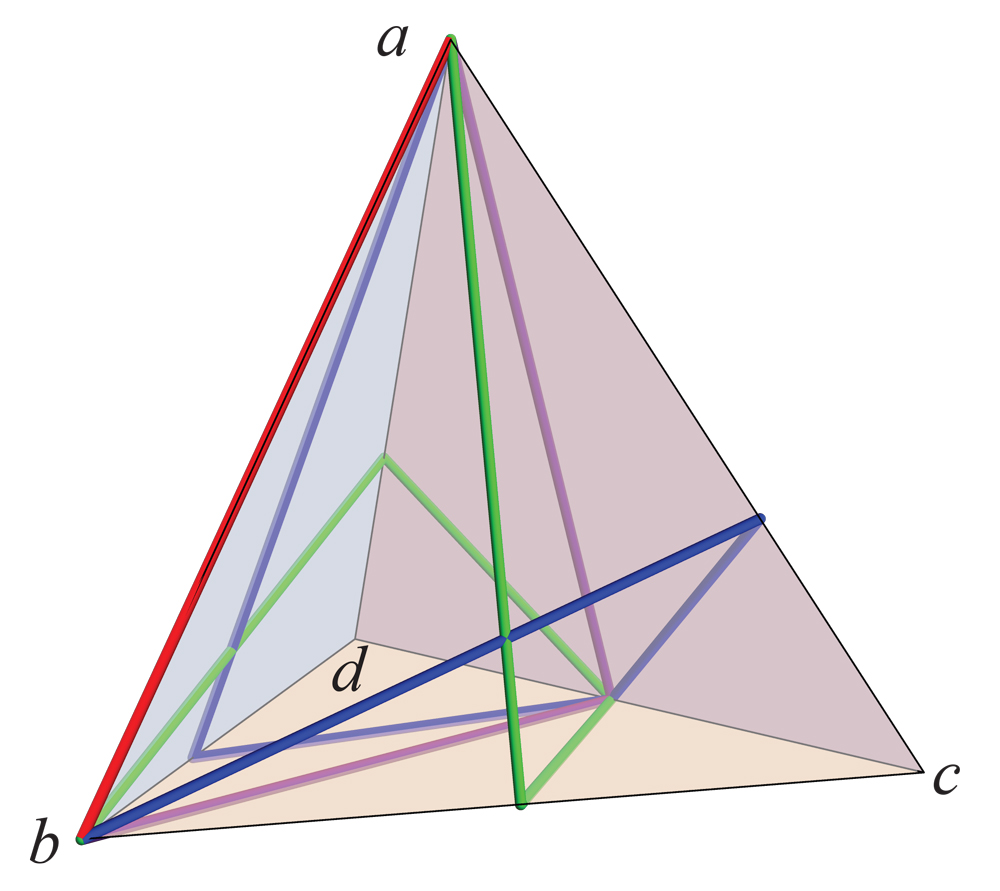}
\caption{Three non-degenerate and one degenerate (red) $Q_2$
on a regular tetrahedron.
The blue, green, and purple geodesic segments each connect to $ab$.}
\figlab{Q2_alt_reg_3D}
\end{figure}


\section{Remarks and Open Problems}
\seclab{Open}

Our work leaves open several questions of various natures.

\bs

\textbf{Open Problem~1.} The $2$-vertex quasigeodesics that we identified in
Section~\secref{Q2} are all edge-loops, i.e., they contain the edge joining the respective vertices.
Is this necessarily the case?
\bs

According to Theorem~\thmref{34}, some tetrahedra have at least $34$ simple closed quasigeodesics,
and this happens on an open subset of the space ${\cal T}$ of tetrahedra.

\bs

\textbf{Open Problem~2.} 
Does there exist an upper bound on the number of simple closed quasigeodesics a tetrahedron can have?
Of course, this is not the case for pure simple closed geodesics, see Section~\secref{Q0}.

\bs

\textbf{Open Problem~3.}  Find examples of tetrahedra with $k \geq 3$ simple closed quasigeodesics, for as many values of $k$ as possible.
For example, is there any tetrahedron that has only the $k=3$ simple closed quasigeodesics that
Pogorelov guarantees and we describe in Theorem~\thmref{Q123}?
Such a tetrahedron would be a polyhedral counterpart of an ellipsoid.
\bs

Our quest for simple closed quasigeodesics on tetrahedra lead us to investigate the acuteness of face angles incident to a vertex.
In this direction, we established the next elementary result, of some independent interest.

\begin{prop}
\lemlab{face_angles_longest_edge}
Let $\bar{e}$ be a longest edge of the tetrahedron $T$. Then at least one extremity of $\bar{e}$ has all incident face angles acute.
\end{prop}

Also notice that the face angles incident to a vertex $v$ of $\omega_v > \pi$ are all acute, directly from Lemma~\lemref{trineq}.

\begin{proof}
Assume $T=abcd$ and $\bar{e}=ab$. Then $ab$ is in particular longest edge in the triangle faces $C$ and $D$, hence the angles $aC, aD, bC, bD$ are all acute.

Assume now that the statement doesn't hold, hence $aB \geq \pi$, $bA \geq \pi$.
Unfold the union of faces $A \cup B$ in the plane, to a quadrilateral $a'b'c'd'$.
Clearly, the triangles $a'c'd'$ and $acd$ are congruent, as are $b'c'd'$ and $bcd$. However, $|a'b'| \geq |ab|$.
The angle conditions $aB \geq \pi$, $bA \geq \pi$ imply, via an elementary geometry result, that the points $a'$ and $b'$ lie on, or in the interior of, the circle of diameter $cd$.
Therefore, we get $|a'b'| \leq |cd| \leq |ab| \leq |a'b'|$, impossible unless we have equalities everywhere.
In this case, $T$ is a doubly covered rectangle, and the conclusion holds. 
\end{proof}

Our proofs involve the vertex of $T$ of largest curvature.

\bs

\textbf{Open Problem~4.} 
Is the longest edge of a tetrahedron always incident to the vertex of largest curvature?
This is indeed the case for degenerate tetrahedra, which correspond to planar quadrilaterals.


\bibliographystyle{alpha}
\bibliography{refs}

\begin{thebibliography}{DDTY17}

\bibitem[AAOS97]{aaos-supa-97}
Pankaj~K. Agarwal, Boris Aronov, Joseph O'Rourke, and Catherine~A. Schevon.
\newblock Star unfolding of a polytope with applications.
\newblock {\em SIAM J. Comput.}, 26:1689--1713, 1997.

\bibitem[Ale55]{a-digdk-55}
Aleksandr~D. Alexandrov.
\newblock {\em Die innere Geometrie der konvexen Fl{\"a}chen}.
\newblock Akademie-Verlag, Berlin, Germany, 1955.

\bibitem[AO92]{ao-nsu-92}
Boris Aronov and Joseph O'Rourke.
\newblock Nonoverlap of the star unfolding.
\newblock {\em Discrete Comput. Geom.}, 8:219--250, 1992.

\bibitem[AP18]{akopyan2018long}
Arseniy Akopyan and Anton Petrunin.
\newblock Long geodesics on convex surfaces.
\newblock {\em Mathematical Intelligencer}, 40:26--31, 2018.

\bibitem[CdM22]{ChartierArnaud}
Jean Chartier and Arnaud de~Mesmay.
\newblock Finding weakly simple closed quasigeodesics on polyhedral spheres.
\newblock To appear, \emph{Symp. Comput. Geom.}, June 2022.

\bibitem[DDTY17]{davis2017geodesics}
Diana Davis, Victor Dods, Cynthia Traub, and Jed Yang.
\newblock Geodesics on the regular tetrahedron and the cube.
\newblock {\em Discrete Mathematics}, 340(1):3183--3196, 2017.

\bibitem[DHK21]{demaine2020finding}
Erik~D. Demaine, Adam~C. Hesterberg, and Jason~S. Ku.
\newblock Finding closed quasigeodesics on convex polyhedra.
\newblock {\em Discrete Comput. Geom.}, 2021.
\newblock arXiv:2008.00589. SoCG2020. To appear \emph{Discrete Comput. Geom.}

\bibitem[DO07]{do-gfalop-07}
Erik~D. Demaine and Joseph O'Rourke.
\newblock {\em Geometric Folding Algorithms: Linkages, Origami, Polyhedra}.
\newblock Cambridge University Press, 2007.
\newblock \url{http://www.gfalop.org}.

\bibitem[Gal03]{gal2003convex}
Gregorii~Aleksandrovich Gal'perin.
\newblock Convex polyhedra without simple closed geodesics.
\newblock {\em Regular Chaotic Dynamics}, 8(1):45--58, 2003.

\bibitem[Gru91]{g-tcscc-91}
Peter Gruber.
\newblock A typical convex surface contains no closed geodesic.
\newblock {\em J. Reine Angew. Math.}, 416:195--205, 1991.

\bibitem[O'R21]{o2021every}
Joseph O'Rourke.
\newblock Every tetrahedron has a 3-vertex quasigeodesic.
\newblock arXiv:2109.07444: \url{https://arxiv.org/abs/2109.07444}, 2021.

\bibitem[OV21]{Reshaping}
Joseph O'Rourke and Costin V\^{i}lcu.
\newblock Reshaping {C}onvex {P}olyhedra.
\newblock arXiv 2107.03153: \url{https://arxiv.org/abs/2107.03153}, July 2021.

\bibitem[Pog49]{p-qglcs-49}
Aleksei~V. Pogorelov.
\newblock Quasi-geodesic lines on a convex surface.
\newblock {\em Mat. Sb.}, 25(62):275--306, 1949.
\newblock English transl., {\em Amer. Math. Soc. Transl.} 74, 1952.

\bibitem[Pro07]{protasov2007closed}
V.~Yu. Protasov.
\newblock Closed geodesics on the surface of a simplex.
\newblock {\em Sbornik: Mathematics}, 198(2):243, 2007.

\bibitem[RV12]{rouyer2012sets}
Jo{\"e}l Rouyer and Costin V{\^\i}lcu.
\newblock Sets of tetrahedra, defined by maxima of distance functions.
\newblock {\em Analele Universitatii ``Ovidius" Constanta-Seria Matematica},
  20(2):197--212, 2012.
\newblock arXiv:1906.11965.

\bibitem[SL92]{strantzen1992regular}
John Strantzen and Yang Lu.
\newblock Regular simple geodesic loops on a tetrahedron.
\newblock {\em Geometriae Dedicata}, 42(2):139--153, 1992.

\end{thebibliography}

\end{document}